\documentclass[letterpaper, 12pt]{article}

\usepackage[english]{babel}
\usepackage{ifthen}
\usepackage{amsmath}
\usepackage{amsthm}
\usepackage{tikz}
\usepackage{relsize}
\usepackage{hyperref}

\usepackage[utf8]{inputenc}
\usepackage[T1]{fontenc}
\usepackage{textcomp}
\usepackage{babel}

\usepackage[backend=bibtex]{biblatex}
\usepackage{import}
\usepackage{pdfpages}
\usepackage{transparent}

\usepackage{amsfonts, amssymb}  %
\usepackage{enumerate}
\usepackage[all]{xy}
\usepackage{wrapfig}
\usepackage{fancyvrb}
\usepackage{listings}

\usepackage{centernot}
\usepackage{mathtools}

\newcounter{casenum}

\newtheorem{thm}{Theorem}[section]

\newtheorem{lem}[thm]{Lemma}

\newtheorem{cor}[thm]{Corollary}
\newtheorem{conj}[thm]{Conjecture}

\newtheorem*{thm*}{Theorem}

\DeclareMathOperator{\N}{N}

\DeclareMathOperator{\Span}{Span}

\title{Real Mahler Series}

\author{Andean E. Medjedovic}

\date{}

\begin{document}

\maketitle

\begin{abstract}
	Let $\alpha$ and $\beta$ be positive real numbers. Let $F(x) \in K[[x^\Gamma]]$ be a Hahn series.
	We prove that if $F(x)$ is both $\alpha$-Mahler and $\beta$-Mahler then it must be a rational function, $F(x) \in K(x)$, assuming
	some non-degeneracy conditions on $\alpha$ and $\beta$.
\end{abstract}

\setcounter{tocdepth}{1}
\tableofcontents
\newpage

\section{Introduction}
During the 1920's Mahler was able to show that the function
\begin{equation}{\label{ex_1}}
	F(x) = \sum_{n=0}^{\infty} x^{2^n}
\end{equation}
takes on transcendental values for every algebraic number $0 < x < 1$ \cite{mahtrans}. Since then his work has been generalized and the \textit{Mahler Method}
is now a staple of transcendence theory \cite{mahstaple}. Further work by Cobham has highlighted connections to automata theory \cite{cob1} \cite{cob2}. Cobham was able to prove
the following.
\begin{thm}[Cobham]
	Suppose $k$ and $l$ are multiplicatively independent integers. Then if a sequence of natural numbers is both $k$-automatic and $l$-automatic, it
	is eventually periodic.
\end{thm}

The proof is infamous for its difficulty despite using only elementary techniques. The correspondence between $k$-automatic sequences and what are known as Mahler functions has been
widely known for at least 15 years \cite{cob3}. Cobham's theorem naturally leads to the analogous question for Mahler functions, which we will now introduce. The astute reader will immediately
verify that the series in equation \ref{ex_1} satisfies the below functional equation:
\begin{equation}
	F(x^2) = F(x) - x.
\end{equation}
We say that particular $F(x)$ is $2$-Mahler of degree $1$. Generalizing this notion, suppose a series satisfies
\begin{equation}
	P_d(x)F( x^{k^d}) + \ldots + P_1(x)F\left( x^k \right) + P_0(x)F\left( x \right)  = A(x).
\end{equation}
We will then say that $F$ is $k$-Mahler of degree $d$. Here $P_i$ and $A$ are polynomials. An equivalent formulation would be to say that
\begin{equation}
	\{1, F(x), F(x^k), F(x^{k^2}), \ldots\}
\end{equation}
forms a $d$-dimensional vector space over $K[x]$.\\

Let $\alpha > 0$ and $\beta > 0$ be real numbers. The reader will recall that if $\frac{\log(\alpha)}{\log(\beta)} \not\in \mathbb{Q}$ then $\alpha$
and $\beta$ are said to be multiplicatively independent. Loxton and Van der Poorten conjectured the following counterpart of Cobham's theorem for Mahler functions
\begin{conj}[Loxton and Van der Poorten]
	Let $k$ and $l$ be multiplicatively independent positive integers. Let $F \in K[[x]]$ be a power series over a number field $K$ that is both $k$-
	and $l$-Mahler. Then $F$ is rational.
\end{conj}
Notice that the converse of the theorem holds trivially, rational functions are $\alpha$-Mahler for any $\alpha$.
In 2013 this conjecture was resolved by Bell and Adamczewski who were the first to prove the more general theorem, over any field of characteristic $0$ \cite{Bell19} .
\begin{thm}[Bell and Adamczewski]
	Let $k$ and $l$ be multiplicatively independent positive integers. Let $F \in K[[x]]$ be a power series over a field of characteristic $0$, $K$, that is both $k$-
	and $l$-Mahler. Then $F$ is rational.
\end{thm}
The theorem was also proven by Schafke and Singer through entirely different methods in 2017 \cite{Singer}.
In this paper we will extend the result beyond positive integers. The main results are as follows.

\begin{thm}{{\label{main}}}
	Let $\alpha>0$ and $\beta>0$ be two multiplicatively independent elements of $\mathbb{Q}$. Suppose $F(x)$ is a Hahn series over $\mathbb{R}$ that is both $\alpha$- and $\beta$-Mahler.
	Then $F(x)$ is rational.
\end{thm}

\begin{thm}{{\label{main2}}}
	Let $\alpha>0$ and $\beta>0$ be two algebraically independent elements of $\mathbb{R}$. Suppose $F(x)$ is a Hahn series over $\mathbb{R}$ that is both $\alpha$- and $\beta$-Mahler.
	Then $F(x)$ is rational.
\end{thm}

\section{A proof sketch}

Here we outline the structure of the argument giving a summary of each section.

\begin{enumerate}
	\item In the next section we will define and summarize some basic facts about Hahn series. We then go on to prove a few quick lemmas that allow us to
		reduce the form of the Mahler functional equations to something more manageable, at least in the rational case. Later on in the proof we will use some more reductions
		of a similar nature that will be proven as we need them.
	\item Afterward, we will restrict our vision for a moment, and look at only Laurent series instead of Hahn series. The goal of this section is to prove a nonexistence result
		of $\alpha$-Mahler Laurent series when $\alpha$ is not an integer. This result, while somewhat interesting in its own right, will allow
		us to make further claims when dealing with the more general case of Hahn series.
	\item We will analyze the supports of Hahn series under the Mahler functional equations to gain a broad understanding of how they interact. This will then prove invaluable
		in making further reductions to a more tractable problem.
	\item Finally, we will tackle the problem and prove theorem \ref{main} under the assumption that $\alpha$ and $\beta$ are rational. In the case
		that $\alpha$ and $\beta$ are integers we appeal to the Adamczewski-Bell theorem. In the other case we aim to reduce the Hahn series
		by putting together all the results found in the previous sections.
	\item Some additional details are needed when one of $\alpha$ or $\beta$ is irrational. We delineate these cases and prove theorem \ref{main2} for them.
		This completes the paper.
\end{enumerate}

\section{Preliminaries}

In this section we introduce some definitions and notation.
We will also take the time to prove some elementary facts and reductions that will be helpful to us later on.
Let $K$ be a number field with $K[x], K(x), K[[x]],$ and $K[[x]][x^{-1}]$ denoting polynomials, rational functions, power series and Laurent series over $K$. %
We will use $(p,q)$ as a shorthand for the greatest common denominator of integers $p$ and $q$.

\subsection{Hahn series}
Let $\Gamma$ be an ordered group. Consider the set of formal expressions, $F,$ so that
\begin{equation}
	F(x) = \sum_{i \in \Gamma}^{} f_i x^i
\end{equation}
with $f_i \in K$ and the additional constraint that the support of $F$ is well-ordered. We will use $P(F)$ to designate this set of non-zero powers appearing in $F$,
\begin{equation}
	P(F) := \{ i \in \Gamma : f_i \ne 0\}.
\end{equation}
Such a formal expression is known as a Hahn series or sometimes as a Hahn-Mal'cev-Neumann series. Of course, one can add Hahn series together in the natural
way. We require $P(F)$ to be well-ordered to ensure that multiplication of Hahn series is well defined. If $F = \sum_{i \in \Gamma}f_i x^i$ and $G = \sum_{i \in \Gamma}g_i x^i$ are two Hahn series then
\begin{equation}
	F(x)G(x) = \sum_{i \in \Gamma} \left( \sum_{j_1+j_2 = i} f_{j_1}g_{j_2}\right) x^{i}.
\end{equation}
Since $P(F)$ and $P(G)$ are well-ordered the inner sum has finitely many non-zero terms and is therefore well-defined. The field of all Hahn series over $K$ with group $\Gamma$ will be indicated
by $K[[x^\Gamma]]$. This field has the following valuation which we will sometimes make use of
\begin{equation}
	v(F) := \min_{ i \in P(F)} i.
\end{equation}
Unless otherwise stated, we default to $\Gamma = \mathbb{R}$.

\subsection{Reductions in the Rational case}
Suppose for a moment that $\alpha = \frac{p_1}{q}$ and $\beta = \frac{p_2}{q}$ are rationals, and we do not necessarily have $\left(p_1, q  \right) =
\left( p_2,q \right) = 1$. Let $F$ be both $\alpha$- and $\beta$-Mahler as a Hahn series. We show in this subsection that it is possible to obtain from $F(x)$,
some other Hahn series with desirable properties.
\begin{gather}{\label{defmah}}
	P_d(x)F( x^{\alpha^d}) + \ldots + P_1(x)F\left( x^\alpha \right) + P_0(x)F\left( x \right)  = A(x)\\
	Q_d(x)F( x^{\beta^d}) + \ldots + Q_1(x)F\left( x^\beta \right) + Q_0(x)F\left( x \right)  = B(x)
\end{gather}

\begin{lem}{\label{homo}}
	We can re-write the functional equations \ref{defmah} as $\alpha$- and $\beta$-Mahler with
	$A(x)$ and $B(x)$ = 0.
\end{lem}
\begin{proof}
	Apply the operator $x \mapsto x^{\alpha q}$ to the first line of equation \ref{defmah}. Multiply both sides by $A(x)$.
	Take the original equation and multiply it by $A(x^{\alpha q})$. These operations take polynomials to polynomials and if we let $F(x^q) = G(x)$ then
	we are left with $2$ Mahler equations for $G$ with equal right hand sides, namely, $A(x^{\alpha q})A(x)$. Take the difference between the two equations
	and we are left with a homogeneous $\alpha$-Mahler equation for $G$. Repeat the argument for the second equation to force $B(x) = 0$.
\end{proof}

\begin{lem}
	If $F$ satisfies equation \ref{defmah} then there
	is a Mahler function, $G$ where we can assume $P_0(x)  \ne 0$.
\end{lem}
\begin{proof}
	Suppose $P_0(x) = P_1(x) = \ldots = P_{i-1}(x) = 0$ and $P_i(x) \ne 0$. Define $G(x)$ to be the Hahn series with $F(x) = G(x^{\frac{1}{\alpha^i}})$.
	Then $G(x)$ satisfies an $\alpha$-Mahler equation with non-zero initial term.
\end{proof}

\begin{lem}{\label{inv}}
	If $F$ satisfies equation \ref{defmah} we can assume $\alpha > 1$ and $\beta > 1$.
\end{lem}
\begin{proof}
	Let $\alpha = \frac{p}{q} < 1$ written in lowest common form. Let $G(x) = F(x^{p^d})$. Apply $x \mapsto x^{\frac{1}{\alpha^d}}$ to the first equation in \ref{defmah}.
	This give a $\frac{1}{\alpha}$-Mahler equation for $F$ over the coefficient ring $K[x^{\frac{1}{\alpha^d}}]$. Taking this equation under the operator $x \mapsto x^{p^d}$
	takes $F$ to $G$ while taking $K[x^{\frac{1}{\alpha^d}}]$ to $K[x]$. Since $\frac{1}{\alpha} > 1$, this will suffice. The proof for $\beta$ is the exact same
	applied to the second. We can guarantee the two series $G$ will be equal by raising to the greatest common denominator of the numerators of $\alpha$ and $\beta$, $G(x) \mapsto G(x^a)$,
	in each case.
\end{proof}

\section{Non-integer Mahler power series}

Let $\alpha = \frac{p}{q}$ be a non-integer positive rational with $\left( p,q \right) = 1$. Require that $\alpha \ne \frac{1}{n}$ for an integers $n$.
Then we will show here that $F$ is a $\alpha$-Mahler Laurent series if and only if $F$ is rational. Notice that the backward direction is trivial.
The restriction on $\alpha$ is quite necessary. If $n$, a natural number, is multiplicatively dependent to $\frac{1}{n}$ we can construct
$\frac{1}{n^k}$-Mahler functions from $n$-Mahler functions. We give one such example and leave the rest to the reader.\\

Let $F(x) = \sum_{n=2}^{\infty} x^{2^n}$. Then $F(x)$ is $2$-Mahler and satisfies
\begin{equation}
	F(x^2) = F(x)-x^4.
\end{equation}
Normalize so that we get a homogeneous functional equation
\begin{equation}
	F(x^4) - (1+x^4)F(x^2) + F(x) = 0
\end{equation}
And applying the $x \mapsto x^{\frac{1}{2}}$ operator twice:
\begin{equation}
	F(x) - (1+x)F(x^{\frac{1}{2}}) + xF(x^{\frac{1}{4}}) = 0
\end{equation}
yields a $\frac{1}{2}$-Mahler functional equation.\\

Moving on to a proof, we say a Laurent series, $F$, in $K[[x]][x^{-1}]$ is $\alpha$-Mahler of degree $d$ if there exist polynomials $P_0(x),\ldots, P_d(x)$
in $K[x]$ so that
\begin{equation}{\label{def_1}}
	P_d(x)F(x^{{\alpha}^{d}})+P_{d-1}(x)F(x^{\alpha^{d-1}})+\ldots+P_0(x)F(x) = 0.
\end{equation}
under these conditions we must have
\begin{thm}{\label{ez}}
	$F(x)$ is in $K(x)$. That is, $F$ is rational.
\end{thm}

The approach is to first reduce to the degree $1$ case.
From there we argue using an infinite product formulation of the problem.

\begin{lem}
	We must have $F(x) = G(x^{q^d})$ for some Laurent series $G(x)$.
\end{lem}
\begin{proof}
	Consider coefficients of terms with exponents of form $a + k\left(\frac{p}{q}\right)^d$ in $\ref{def_1}$ for $a,k \in \mathbb{N}$ with $q \nmid k$. They have to be $0$
	from the right hand side and on the left hand side they only come from the $P_d(x)F(x^{\alpha ^{d}})$ term. Thus $F(x) = F_0(x^q)$ for some Laurent series $F_0$.
	Now substitute $F_0(x^q)$ for $F(x)$ in \ref{def_1} and repeat the argument mutatis mutandis.
\end{proof}

\begin{cor}
	$F(x^{\alpha ^{i}})$ are Laurent series as well for $i = 0 ,1 \ldots, d$.
\end{cor}

We use $\mathbb{Z} + \gamma$ to denote the set $\{z + \gamma\}$ for $z \in \mathbb{Z}$.

\begin{lem}
	Let $F(x)$ be a Laurent series and let $S(x),T(x)$ be Hahn series. Let
	\[
		S(x) = T(x)F(x)
	.\] If $_{\mathbb{Z}+\gamma}\left[ S(x) \right]$ denotes the sum of terms composing $S(x)$ with exponents in $\mathbb{Z}+\gamma$ then
	\[
		_{\mathbb{Z} +\gamma}\left[S(x)  \right] = _{\mathbb{Z}+\gamma}\left[ T(x) \right] F(x)
	.\]
\end{lem}
\begin{proof}
	Since $F(x)$ is a Laurent series, the exponents of terms in $F(x)$ are in $\mathbb{Z}$. So each term of form $x^\gamma$ maps to a sum of terms of form $x^{z +\gamma}$ after
	multiplication by $F(x)$.
\end{proof}

We will later use the more abstract version of this notation. If $G$ is a Hahn series we will take ${}_A[G]$ to mean the Hahn series containing only terms with powers contained in $A$. An immediate consequence of this definition is $P({}_A[G]) = A\cap P(G)$.\\

We are ready to prove the central result of this section and let $\alpha >1$.
Write $F(x) = x^{v(F) - 1}G(x)$, then $G(x^l)$ is also $\alpha$-Mahler for some $l$
that divides $v(Q_i)$ for polynomials in the functional equation defining
If we know
\begin{equation}
	Q_d(x)G(x^{{\alpha}^{d}})+Q_{d-1}(x)G(x^{\alpha^{d-1}})+\ldots+Q_0(x)G(x) = 0
\end{equation}
then we know the powers of $x$
across the left hand side must appear in $2$ of the $d$ terms
and cancel to $0$. In particular, the minimal terms in each of the $d$ summands
must satisfy this constraint so we know
\[
\alpha^n l + b_0 = \alpha^m l + b_1
\]
for some integers $b_0$ and $b_1$ and naturals $n$ and $m$.
If $\alpha$ is a rational that is not an integer then
\[
\alpha^n - \alpha^m \not\in \mathbb{Z}
\]
while $\frac{b_1-b_0}{l}$ is. Contradiction. \qed

\section{Reduction to linear subsets}

Let $S$ be a well ordered subset of $\mathbb{R}$ under $<$.
Consider the formal Hahn series $F \in K^S[[x]]$.
Suppose $F$ is a homogeneous $\alpha$-Mahler function as well, with $\alpha \in \mathbb{R}$ for some $\alpha > 1 \in \mathbb{R}$.
Again, by Lemma \ref{inv}, $\alpha > 1$ is assumed to be the case.\\

Let $P_d(x) \ne 0$ and assume $F$ satisfies
\begin{equation}\label{mah}
	\sum_{i=0}^{d} P_i(x)F(x^{{\alpha}^i}) = 0
\end{equation}

where $F(x) = \sum_{s\in S} f_s x^s$.
We can assume $f_s$ is never $0$ by taking a subset of $S$ if need be.
$S$ must be countable, and use the well-ordering principle to index $S = \{s_0, s_1,\ldots \}$.\\

The point of this section is to establish this theorem.
\begin{thm}\label{decomp}
	Let $F$ be an $\alpha$-Mahler Hahn series. Then,
	using the notation instituted earlier, we can rewrite $F$ as a sum of
	other $\alpha$-Mahler functions corresponding to each equivalence class

	\begin{equation}
		F(x) = \sum_{s \in \widehat{P(F)}} {}_{T(s)}[F(x)]
	\end{equation}
	where $T(s) := \alpha^\mathbb{Z} s+\mathbb{Z}[\alpha][\alpha^{-1}]$.

\end{thm}
We use this theorem later on in the proof of our main theorem to reduce to the case where the
set of powers of $x$ is a subset of $\mathbb{Z}[\alpha][\alpha^{-1}]$.\\

\begin{proof}
	For now assume $F$ satisfies equation \label{mah} and write $F$ as
	\begin{equation}{\label{rat}}
		F(x) = \sum_{i=1}^{d}R_i(x)F(x^{\alpha^{i}})
	\end{equation}
	for some rational functions $R_i(x)$. Notice the exponents of terms of the right hand side all of form
	\begin{equation}
		\alpha^i s + c
	\end{equation}
	where $c \in \mathbb{Z}$ and $s \in P(F)$. Every exponent occurring on the left must also occur on the right. In this case the left is
	$P(F)$ while the exponents occurring on the right will be a union of the sets $\alpha^i P(F) + c$ as $i$ and $c$ vary. Then
	\begin{equation}{\label{subset}}
		P(F) \subset \bigcup_{n=1}^{\infty}\alpha^n P(F) + \mathbb{Z}[\alpha][\alpha^{-1}]
	\end{equation}
	where $A+B$ is the set $a+b$ over all $a \in A$ and $b \in B$. Let $x$ and $y$ be elements of $\mathbb{R}$ and we will define the equivalence relation $\sim$
	to mean there is some $p \in \mathbb{Z}[\alpha][\alpha^{-1}]$ and integer $m$ for which
	\begin{equation}
		\alpha^m x + p = y.
	\end{equation}
	The reader can check that this is indeed an equivalence relation.
	The upshot of this equivalence is that we can pass it through equation \ref{subset}.
	Let $\widehat{P(F)} = P(F) / \sim $ so that we now have
	\begin{equation}
		P(F) \subset \bigcup_{n=1}\alpha^n\widehat{P(F)} + \mathbb{Z}[\alpha][\alpha^{-1}].
	\end{equation}
	Notice that each equivalence class $\alpha^\mathbb{Z}s + \mathbb{Z}[\alpha][\alpha^{-1}]$ is closed
	under multiplication by $\alpha$ and addition of integers.
	Therefore, we can rewrite $F$ as a sum over a $\alpha$-Mahler function corresponding to each equivalence class
	\begin{equation}
		F(x) = \sum_{s \in \widehat{P(F)}} {}_{T(s)}[F(x)]
	\end{equation}
	where $T(s) := \alpha^\mathbb{Z} s+\mathbb{Z}[\alpha][\alpha^{-1}]$.
	In which case ${}_{T(s)}[F]$ is $\alpha$-Mahler with the same polynomial scalars, as required.
	Note that each term in the sum is $\alpha$-Mahler with the same functional equation.\\
\end{proof}

	We shall soon see that we can sharpen the restriction $P(F) \subset \mathbb{Z}[\alpha][\alpha^{-1}] +\alpha^\mathbb{Z}s$ further still, under some assumptions.
	We will however, wait and
	do this spread out over a few cases.
	For now, we point out another restriction, namely

\begin{lem}{\label{finiteness}}
	Let $F$ be an $\alpha$-Mahler Hahn series.
	$P(F) \cap T(s) = \emptyset$ for all but finitely many $s \in \mathbb{R}$. In other words, $\widehat{P(F)}$ is a finite set.
\end{lem}
\begin{proof}
	Suppose $F$ satisfies
	\begin{equation}
		\sum_{i=0}^d P_i(x) F(x^{\alpha^i}) = 0.
	\end{equation}
	Consider the valuation of each term in the sum, $v(P_i(x)F(x^{\alpha^i})) = c_i + \alpha^i f_0$ where $v(P_i) = c_i$ and $v(F) = f_0$.
	For the right hand side to be $0$, terms on the left hand side must all cancel out. In particular, we must have solutions to
	\begin{equation}
		c_j + \alpha^j f_0 = c_i +\alpha^i f_0
	\end{equation}
	for some $i \ne j \in \{0, \ldots, d\}$.
	Here $i,j$ are allowed to vary while the $c_i, c_j$ and $\alpha$ are fixed by the functional equation. Of course, each equation is linear so there are only
	finitely many $f_0$ that are possible as $i$ and $j$ vary. From this we conclude $T(s) \cap P(F) = \emptyset$ for all but finitely many $s \in \mathbb{R}$.
\end{proof}

	We end this section we a small observation of what occurs in the case where one of the $_{T(s)}[F]$ is rational, which will be useful to us.
\begin{lem}{\label{cover}}
	Suppose $F(x)$ is an $\alpha$-Mahler Hahn series with $\alpha$ rational.
	There is some integer $l$ for which $P(F(x^l)) \subset \mathbb{Z}[\alpha, \alpha^{-1}]$.
	Trivially, $F(x^l)$ is an $\alpha$-Mahler Hahn series as well.
\end{lem}
\begin{proof}
	Let $s \in \mathbb{R}$ be arbitrary. There are only finitely many $s$ for which $P(F) \cap T(s) \ne \emptyset$ by Lemma \ref{finiteness}.
	Note that $\sim$ is an equivalence on all of $\mathbb{R}$, so any element of $P(F)$ is contained in a member of it. That is, the union
	of all $T(s)$ cover $P(F)$.
	By Lemma \ref{decomp1} we may further assume $s \in \mathbb{Q}$. We have finitely many rationals $s$ for which the intersection with the support is non-trivial, let
	$l$ be a multiple of the least common multiple of all such $s$. Then $l\times T(s) \subset \mathbb{Z}[\alpha, \alpha^{-1}]$. From this
	it follows $P(F(x^{l})) \subset \mathbb{Z}[\alpha, \alpha^{-1}]$.
\end{proof}

\section{Doubly rational Mahler Hahn series}
As before let $F(x) \in K^{\mathbb{R}}[[x]]$ be a Hahn series over $\mathbb{R}$. Let $\alpha = \frac{p_1}{q}$ and $\beta = \frac{p_2}{q}$ be multiplicatively independent
rationals (although not necessarily written in lowest common form).
In this section we aim to prove theorem \ref{main}
\begin{thm*}
	If $F(x)$ is both $\alpha$ and $\beta$-Mahler then $F(x)$ is rational.
\end{thm*}

We begin by strengthening the previous Lemma. Our goal is to show we can restrict to where $P(F) \subset \mathbb{Z}[\alpha, \alpha^{-1}]$, and similarly for
$\beta$. Decompose $F$ across equivalence classes as in theorem \ref{decomp}.
By focusing on each term of the decomposition individually we can
restrict to where $P(F) \subset T(s)$ for some properly chosen $s$ by Lemma \ref{decomp}.
That is, suppose $P(F) \subset \alpha^\mathbb{Z}s + \mathbb{Z}[\alpha, \alpha^{-1}]$ for some $s \not\sim 0 $.
We show that $s$ must be rational. In the case that $s$ is rational, we can take $F(x)$ to a large enough
power by the operator $x \mapsto x^l$.\\

\begin{lem}
	Let $F$ be $\alpha$-Mahler.
	Then the
	functions $F$ and $_{T(s)}[F]$ satisfy the same $\alpha$-Mahler function.
\end{lem}
\begin{proof}
	The proof is trivial, begin with
	\begin{equation}
	\sum_{n=0}^{d} P_d(x)F(x^{\alpha^i}) = 0
	\end{equation}
	and apply the $_{T(s)}[\cdot]$ to the equation. Since $\mathbb{Z} + T(s) = T(s)$
	it follows that the operator is linear over $K[x]$.
\end{proof}

\begin{lem}{\label{decomp1}}
	Suppose $s$ is irrational. Then $_{T(s)}[F] = 0$.
\end{lem}
\begin{proof}
Suppose $s$ is irrational. We will show that this forces $F = 0$. Notice that if $s$ is irrational then the only solutions to
\begin{equation}
	\alpha^m s + p_1 = \alpha^ns + p_2
\end{equation}
where $m,n$ are integers and $p_1, p_2 \in \mathbb{Z}[\alpha, \alpha^{-1}]$ is when $m = n$ and $p_1 = p_2$. We aim to contradict the well-ordered property of $P(F)$.
We know $F$ must satisfy
\begin{equation}
	P_0(x)F(x) = \sum_{n=1}^{d} P_d(x)F(x^{\alpha^i}).
\end{equation}
We can write every element of $f \in T(s)$ uniquely as $\alpha^n s + p$ with $n \in \mathbb{Z}$ and $p \in \mathbb{Z}[\alpha, \alpha^{-1}]$.
We call the pair $(n ,\deg(p))$ the degree of $f$. The degree of $p \in \mathbb{Z}[\alpha, \alpha^{-1}]$ is the largest
power of $\alpha$. Suppose $f \in P(F)$. So a term of form $x^f$ appears in the left hand side of the above equation. Of course, it must also come from a
term on the right hand side so
\begin{equation}
	f = c + \alpha^i f_0
\end{equation}
where $c$ is a term coming from rational scaling terms $P_i$ and $f_0$ is another element of $P(F) \subset T(s)$. Now notice the degree of $f_0$ is either
$(n-i, \deg(p) -i)$ or $(n-i, -i)$ depending on whether the constant $c$ comes into play. Apply the same argument to $f_0$ inductively to get a sequence
of $f_j$ and notice the degree lexicographically decreases every iteration. The sequence of polynomials which give the second argument of the degree either converges to $0$
or becomes arbitrarily close to $\frac{c}{\alpha^i}$ over all possible $c, i$, of which there are finitely many. In the latter case choose a subsequence of the $f_j$ so that
the polynomials converge to a fixed $\frac{c_0}{\alpha^{i_0}}$, in which case $f_j$ must also converge to the same value.\\

Recall that $P(F)$ must be well-ordered by the definition of Hahn series. By the convergence of the $f_j$, $P(F)$ is dense somewhere. The reader will notice
that this is a contradiction. This completes the proof. We move on to the case where $s$ is rational.\\
\end{proof}

\begin{lem}{\label{finaldecompo}}
	Suppose $_{\alpha^\mathbb{Z}s + \mathbb{Z}[\alpha,\alpha^{-1}]}[F] \ne 0$ and $F$ is $\alpha$-Mahler.
	Suppose furthermore $_{\beta^\mathbb{Z}s' + \mathbb{Z}[\beta, \beta^{-1}]}[F] \ne 0$ and $F$ is $\beta$-Mahler.
	Then there is an integer $n$ for which $_{\mathbb{Z}[\alpha, \alpha^{-1}]}[F(x^n)] \ne 0$
	and $_{\mathbb{Z}[\beta, \beta^{-1}]}[F(x^n)] \ne 0$.
\end{lem}
\begin{proof}
We have seen that if $s$ is irrational, then $_{T(s)}[F] = 0$.
Suppose $s = \frac{k}{l}$ is rational.
Apply $x \mapsto x^l$ to $F$ to get a function, $G$,  with $P(G) \subset \mathbb{Z}[\alpha, \alpha^{-1}]$. Do
a similar operation to $F$ for $\beta$. Suppose we decompose $F$  with $_{T(s')}[F] \ne 0$, and
$T(s')= \beta^\mathbb{Z} s' + \mathbb{Z}[\beta, \beta^{-1}]$
where $s' = \frac{k'}{l'}$. Taking $x \mapsto x^{l\cdot l'}$ gives $P(F) \subset \mathbb{Z}[\alpha, \alpha^{-1}]$
and $P(F) \subset \mathbb{Z}[\beta, \beta^{-1}]$.
Then we can proceed to the below proof.\\
\end{proof}

The case where both $\alpha$ and $\beta$ are integers is the Adamczewski-Bell theorem \cite{Bell19}.
Therefore we can assume at least one of $\alpha$ or $\beta$ is not an integer.\\

We quickly deal with the case that $\beta$ is an integer.
Assume that $\beta$ is an integer. By the theorem we just proved, \ref{decomp} ,we have that $P(F) \subset \mathbb{Z}$ in which case $\alpha$ being a non-integer rational and $F(x)$ being $\alpha$-Mahler
contradicts the theorem established two sections ago, theorem \ref{ez} unless $F(x)$ is rational.\\

We begin the main thrust with a Lemma.
\begin{lem}{\label{combo}}
	Fix an $N > 0$. There is an integer $l$ so that,
	if $F(x)$ is $\alpha$ and $\beta$-Mahler then $F(x^l)$ is $\alpha^n \beta^m$-Mahler for all integers $m, n$ with $|m| < N$ and $|n| < N$.
\end{lem}
\begin{proof}
	The idea is the same as in proposition $8.1$ in Adamczewski-Bell \cite{}. We must be more careful to take the powers appearing in the polynomials into account.
	Suppose $F(x)$ is $\alpha$-Mahler of degree $d_1$ and $\beta$-Mahler of degree $d_2$. Consider the functional equations defining $F(x)$:
	\begin{equation}{\label{a}}
		\sum_{i=0}^{d_1}P_i(x)F(x^{\alpha^i}) = 0
	\end{equation}
	\begin{equation}{\label{b}}
		\sum_{i=0}^{d_1}Q_i(x)F(x^{\beta^i}) = 0.
	\end{equation}
	Applying the $x \mapsto x^{\alpha^i}$ to equation gives  a linear dependence for $F(x^{\alpha^{d_1+i}})$ in terms of lower order terms over $K[x^{\alpha^i}]$.
	Using the formula for the lower order terms gives $F(x^{\alpha^{d_1+1}})$ in the span
	\begin{equation}{\label{pozzed}}
		F(x^{\alpha^{d_1+1}}) \in \Span_{K[x^{\alpha^i}]}\{F(x) , \ldots , F(x^{\alpha^{d_1-1}})\}.
	\end{equation}
	Of course, the same fact holds for $\beta$ after making the necessary changes. In case, we are interested in $F(x^{\alpha^m})$ for $m < 0$ apply $x \mapsto x^{\alpha^m}$.
	This gives a linear dependence of for $F(x^{\alpha^m})$ but in terms of higher degree terms over the coefficient field $K[x^{\alpha^m}]$. Inductively,
	we conclude the negative version of equation \ref{pozzed}:
	\begin{equation}
		F(x^{\alpha^{m}}) \in \Span_{K[x^{\alpha^m}]}\{F(x) , \ldots , F(x^{\alpha^{d_1-1}})\}.
	\end{equation}
	The same holds for $\beta$. Let $$V_l = \Span_{K[x^{\frac{1}{l}}]}\{F(x^{\alpha^i \beta^j})\}_{\substack{0 \leq i \leq d_1-1 \\ 0 \leq j \leq d_2-1}}.$$
	Notice that if $l$ is sufficiently large then by the above logic we know $F(x^{\alpha^{n_0} \beta^{m_0}})$ is in $V_l$ for all $|n_0| <d_1d_2N$ and $|m_0| < d_1d_2N$.
	Since the dimension of $V_l$ is at most $d_1d_2$ we know and $F(x^{\alpha^{jn} \beta^{jm}})$ is in $V_l$ for all $j = 0, \ldots d_1d_2$.\\

	From this we conclude that $F(x)$ is $\alpha^n \beta^m$-Mahler of degree at most $d_1d_2$, over the coefficient ring $K[x^{\frac{1}{l}}]$.
	Finally, applying $x \mapsto x^{l}$ gives the desired conclusion.
\end{proof}

So we consider instead $G(x) = F(x^l)$. We haven't yet specified $N$, and $l$ depends on it.  We will leave it unspecified for now. Let $p$ be a prime and $v_p$, the $p$-adic valuation.
We need one more quick lemma before we embark on the final stretch of the rational case.
\begin{lem}
	There is some natural number $N$ for which
	\begin{equation}
		v_p(\alpha^n \beta^m) = 0
	\end{equation}
	has a solution in $n$ and $m$, for any prime $p$, with $|n|$ and $|m|$ bounded by $N$.
\end{lem}
\begin{proof}
	The proof is trivial; expand $v_p(\alpha^n \beta^m) = nv_p(\alpha)+mv_p(\beta)$ and notice we can take $n = k v_p(\beta), m = -kv_p(\alpha)$ for an integer $k$ larger than, say $d_1d_2$.
	Since only finitely many primes appear in the expansions of $\alpha$ or $\beta$, a finite $N$ will indeed exist.
\end{proof}

Moving on to the proof of theorem \ref{main}.

\begin{lem}
	Let $N$ be a natural number.
	Let $F(x)$ an arbitrary $\alpha$-Mahler Hahn series with $\alpha$ rational.
	We can choose an integer $l$ for which the support of $G(x) = F(x^l)$ is
\begin{equation}
	P(G) \subset \bigcap_{|n|,|m| \leq N} \mathbb{Z}[\alpha^n \beta^m, \left(\alpha^n \beta^m\right)^{-1}].
\end{equation}
\end{lem}

\begin{proof}
	By Lemma \ref{combo} there is an integer $l_0$ for which $F(x^{l_0})$ is $\alpha^n \beta^m$-Mahler for all integers $m, n$ with $|m| < N$ and $|n| < N$.
	By Lemma \ref{cover}, for each of the $M$ different $\alpha^n \beta^m$-Mahler functional equations we obtained ($n$ and $m$ vary)
	there is another integer $l_i$ ($i$ goes from $1$ to $M$) for which
	$P(F(x^{l_0 l_i})) \subset \mathbb{Z}[\alpha^n \beta^m , \alpha^{-n} \beta^{-m}]$.
	Let
	\[
		l = \prod_{i= 0}^{M}l_i.
	\]
	Then if $G(x) = F(x^l)$ we have
	\[
	P(G) \subset \bigcap_{|n|,|m| \leq N} \mathbb{Z}[\alpha^n \beta^m, \left(\alpha^n \beta^m\right)^{-1}]
	\]
	as required.

\end{proof}

Using the above Lemma we can complete the proof in a few lines.
Take any prime, $p$. By construction, there is an element of the intersection, $\mathbb{Z}[\alpha^{n'} \beta^{m'}]$ which does not contain $\frac{1}{p}$.
From this we conclude
\begin{equation}
	P(G) \subset \mathbb{Z}.
\end{equation}
But theorem \ref{ez} forbids this. Contradiction.\\

\section{Doubly irrational Mahler Hahn series}
Some care needs to be devoted to the case where one of $\alpha$ or $\beta$ is irrational. For this section we assume $\alpha$ and $\beta$ are algebraically independent.
We demonstrate the analogous theorem to theorem \ref{main} in this case, theorem \ref{main2}.
\begin{thm*}
	If $F(x)$ is a Hahn series that is both $\alpha$- and $\beta$-Mahler then $F(x) \in K(x)$.
\end{thm*}
\begin{proof}
	We know $F$ satisfies both
	\begin{equation}
		\sum_{i=0}^{d_1}P_i(x)F(x^{\alpha^i}) = 0
	\end{equation}
	and
	\begin{equation}
		\sum_{i=0}^{d_1}Q_i(x)F(x^{\beta^i}) = 0.
	\end{equation}
	Consider the minimal element of $P(F)$, namely, $v(F)$. Over all
	possibilities, there is some
	$l \leq d_1$ and integer $a$ coming from exponents of $x$ in $P_i$
	so that the quantity $\alpha^lv(F)+a$ is minimized. For the left hand side of
	the above to be $0$ the value $\alpha^lm+a$ must occur in two different terms.
	That is to say, there is some $l_1, l_2$ so that
	\[
		(\alpha^{l_1} - \alpha^{l_2})v(F) \in \mathbb{Z}.
	\]
	Mutatis Mutandis for $\beta$ in which case
	\[
		\left( \beta^{k_1} - \beta^{k_2} \right)v(F) \in \mathbb{Z}.
	\]
	But then, of course,
	\[
		\frac{\left( \alpha^{l_1} - \alpha^{l_2} \right)}
		{\left( \beta^{k_1} - \beta^{k_2} \right)} \in  \mathbb{Q}
	\]
	so $\alpha$ and $\beta$ cannot be algebraically independent.

\end{proof}

\section{Concluding Remarks}
To summarize, we have shown that there exist no Laurent series that are $\alpha$-Mahler for $\alpha$ multiplicatively independent from a natural number, other
than rational functions $F \in K(x)$. We then use this result to prove there exist no doubly Mahler Hahn series, $\alpha$- and $\beta$-Mahler Hahn series,
where $\alpha$ and $\beta$ are multiplicatively independent rationals and algebraically independent numbers. A question open for further investigation
is whether this result for irrational numbers can be relaxed to $\alpha$ and $\beta$ multiplicatively independent instead of algebraically independent:\\
\begin{conj}
	Let $\alpha$ and $\beta$ be irrational and multiplicatively independent. If $F(x)$ is a Hahn series over $\mathbb{R}$ that is both
	$\alpha$- and $\beta$-Mahler, then $F(x) \in K(x)$.\\
\end{conj}

\textbf{Acknowledgements. }--- The author would like the thank Jason Bell for suggesting the problem and discussions surrounding the problem. Without him this paper would not
be possible.

\newpage
\nocite{*}
\printbibliography

\end{document}